\UseRawInputEncoding
\documentclass[11pt]{amsart}
\usepackage{url}
\usepackage{amssymb}
\usepackage{verbatim}
\usepackage{hyperref}
\usepackage{color}
\usepackage{mathrsfs}

\begin{document}

\newtheorem{theorem}{Theorem}    
\newtheorem{proposition}[theorem]{Proposition}
\newtheorem{conjecture}[theorem]{Conjecture}
\def\theconjecture{\unskip}
\newtheorem{corollary}[theorem]{Corollary}
\newtheorem{lemma}[theorem]{Lemma}
\newtheorem{sublemma}[theorem]{Sublemma}
\newtheorem{observation}[theorem]{Observation}
\newtheorem{remark}[theorem]{Remark}
\newtheorem{definition}[theorem]{Definition}
\theoremstyle{definition}
\newtheorem{notation}[theorem]{Notation}
\newtheorem{question}[theorem]{Question}
\newtheorem{questions}[theorem]{Questions}
\newtheorem{example}[theorem]{Example}
\newtheorem{problem}[theorem]{Problem}
\newtheorem{exercise}[theorem]{Exercise}

\numberwithin{theorem}{section} \numberwithin{theorem}{section}
\numberwithin{equation}{section}

\def\earrow{{\mathbf e}}
\def\rarrow{{\mathbf r}}
\def\uarrow{{\mathbf u}}
\def\varrow{{\mathbf V}}
\def\tpar{T_{\rm par}}
\def\apar{A_{\rm par}}

\def\reals{{\mathbb R}}
\def\torus{{\mathbb T}}
\def\heis{{\mathbb H}}
\def\integers{{\mathbb Z}}
\def\naturals{{\mathbb N}}
\def\complex{{\mathbb C}\/}
\def\distance{\operatorname{distance}\,}
\def\support{\operatorname{support}\,}
\def\dist{\operatorname{dist}\,}
\def\Span{\operatorname{span}\,}
\def\degree{\operatorname{degree}\,}
\def\kernel{\operatorname{kernel}\,}
\def\dim{\operatorname{dim}\,}
\def\codim{\operatorname{codim}}
\def\trace{\operatorname{trace\,}}
\def\Span{\operatorname{span}\,}
\def\dimension{\operatorname{dimension}\,}
\def\codimension{\operatorname{codimension}\,}
\def\nullspace{\scriptk}
\def\kernel{\operatorname{Ker}}
\def\ZZ{ {\mathbb Z} }
\def\p{\partial}
\def\rp{{ ^{-1} }}
\def\Re{\operatorname{Re\,} }
\def\Im{\operatorname{Im\,} }
\def\ov{\overline}
\def\eps{\varepsilon}
\def\lt{L^2}
\def\diver{\operatorname{div}}
\def\curl{\operatorname{curl}}
\def\etta{\eta}
\newcommand{\norm}[1]{ \|  #1 \|}
\def\expect{\mathbb E}
\def\bull{$\bullet$\ }
\def\C{\mathbb{C}}
\def\R{\mathbb{R}}
\def\Rn{{\mathbb{R}^n}}
\def\Sn{{{S}^{n-1}}}
\def\M{\mathbb{M}}
\def\N{\mathbb{N}}
\def\Q{{\mathbb{Q}}}
\def\Z{\mathbb{Z}}
\def\F{\mathcal{F}}
\def\L{\mathcal{L}}
\def\S{\mathcal{S}}
\def\supp{\operatorname{supp}}
\def\dist{\operatorname{dist}}
\def\essi{\operatornamewithlimits{ess\,inf}}
\def\esss{\operatornamewithlimits{ess\,sup}}
\def\xone{x_1}
\def\xtwo{x_2}
\def\xq{x_2+x_1^2}
\newcommand{\abr}[1]{ \langle  #1 \rangle}

\newcommand{\Norm}[1]{ \left\|  #1 \right\| }
\newcommand{\set}[1]{ \left\{ #1 \right\} }
\def\one{\mathbf 1}
\def\whole{\mathbf V}
\newcommand{\modulo}[2]{[#1]_{#2}}

\def\scriptf{{\mathcal F}}
\def\scriptg{{\mathcal G}}
\def\scriptm{{\mathcal M}}
\def\scriptb{{\mathcal B}}
\def\scriptc{{\mathcal C}}
\def\scriptt{{\mathcal T}}
\def\scripti{{\mathcal I}}
\def\scripte{{\mathcal E}}
\def\scriptv{{\mathcal V}}
\def\scriptw{{\mathcal W}}
\def\scriptu{{\mathcal U}}
\def\scriptS{{\mathcal S}}
\def\scripta{{\mathcal A}}
\def\scriptr{{\mathcal R}}
\def\scripto{{\mathcal O}}
\def\scripth{{\mathcal H}}
\def\scriptd{{\mathcal D}}
\def\scriptl{{\mathcal L}}
\def\scriptn{{\mathcal N}}
\def\scriptp{{\mathcal P}}
\def\scriptk{{\mathcal K}}
\def\frakv{{\mathfrak V}}

%
\newtheorem*{remark0}{\indent\sc Remark}
%
\renewcommand{\proofname}{\indent\sc Proof.} 

\title[Fractional integral operators on mixed $\lambda$-central central  Morrey spaces]
{Fractional integral operators on the mixed $\lambda$-central central  Morrey spaces}
\renewcommand{\thefootnote}{}
\footnotetext[1]{\textcolor{black}{2020} \textit{Mathematics Subject Classification}. Primary 42B20;
 Secondary 42B25.}

%
\thanks{The second author is the corresponding author.}
\thanks{The research was supported by the NNSF of China (No. 12061069.)}
\date{\today}

\author[W. Lu]{Wenna Lu}
\address{Wenna Lu:
College of Mathematics and System Sciences\\ Xinjiang University\\
\"{U}r\"{u}mqi  830046\\ People's Republic of China}
\email{luwnmath@126.com}

\author[J. Zhou]{Jiang Zhou}
\address{Jiang Zhou: College of Mathematics and System Sciences\\ Xinjiang University\\
\"{U}r\"{u}mqi  830046\\ People's Republic of China}
\email{zhoujiang@xju.edu.cn}

\maketitle
\begin{abstract}
In this paper, the authors define the mixed $\lambda$-central  Morrey spaces and the mixed $\lambda$-central $BMO$ spaces. The boundedness of the fractional integral operators  $T_{\alpha}$ and its commutators $[b, T_{\alpha}]$ are established  on the mixed $\lambda$-central Morrey spaces, respectively. Furthermore, we also extend these results to the generalized mixed central Morrey spaces.
\end{abstract}

\section{Introduction}
In the past few decades, the subject of the mixed-norm function spaces has undergone a vast development. Nevertheless, the standard literature is still the mixed Lebesgue spaces $L^{\vec{p}}(\mathbb{R}^{n})(0<\vec{p}\leq\infty)$, as a natural generalization of the classical Lebesgue spaces $L^{p}(\mathbb{R}^{n})(0<p\leq\infty)$, it is first introduced by Benedek and Panzone \cite{BP} in 1961. In fact, the more accurate structure of the mixed-norm function spaces than the corresponding classical function spaces such that the mixed-norm function spaces have more extensive applications in various areas such as potential analysis, harmonic analysis and partial differential equations, see \cite{AI,KC,KD,KN} and the references therein. Therefore, the topic of function spaces with the mixed norm has increasing considerable attention and much developments in recent years. For instance, weak mixed-norm Lebesgue spaces \cite{CS}, mixed-norm Besov spaces and Triebel-Lizorkin spaces \cite{GN,JS}, mixed-norm Lorentz spaces and Orlicz spaces \cite{FD,MM2}, mixed-norm Lorentz-Marcinkiewicz spaces \cite{MM1} and anisotropic mixed-norm Hardy spaces \cite{CGN}, etc.

  It is well known that the classical Morrey spaces, as an important generalization of the Lebesgue space, are introduced by Morrey \cite{CM} in 1938. Up till now, the Morrey spaces have become one of the most important function spaces in the theory of function spaces due to its  potential applications .
  In 2019, Nogayama \cite{TN1}, to generalize Morrey spaces and mixed Lebesgue spaces, introduced the new Morrey type space, which is called the mixed Morrey space.
  Recently, Wei given the definiton of the  mixed central Morrey space in \cite{WM1}. Moreover, these function spaces have interesting applications in studying boundedness of operators including singular integral operators,  maximal operators, the fractional integral operators and their commutators.


Motivated by \cite{CM,WM1,TN1,TN2}, the purpose of this paper is to give the difinitions of the $\lambda$-central Morrey spaces and $\lambda$-central $BMO$ spaces with mixed norm, respectively,  and then we establish the boundedness of fractional integral operator $T_{\alpha}$ and its commutator $[b,T_{\alpha}]$ on the $\lambda$-central Morrey spaces with mixed norm. Furthermore, we also generalize these consequences to the generalized central Morrey spaces with mixed norm.

Let us explain the outline of this article. In Section 2, we first briefly recall the mixed Lebesgue
spaces and the mixed central Morrey spaces. Then we will define the $\lambda$-central $BMO$ spaces and $\lambda$-central Morrey spaces with mixed norm, and obtain some properties of these spaces. In Section 3, we will establish the boundedness of Calder\'{o}n-Zygmund operators, the fractional integral operators  and their commutators on the $\lambda$-central Morrey spaces with mixed norm. In Section 4, we will further extend these results to the generalized central Morrey spaces with mixed norm,

As a rule, for any set $E\in\mathbb{R}^n$, ${\chi}_{E}$ denotes its characteristic function and $E^{c}$ denotes its complementary set, we also denote the Lebesgue measure by $|E|$. Let $\mathscr{M}(\mathbb{R}^{n})$ be the class of Lebesgue measurable functions on $\mathbb{R}^{n}$.  We use the symbol $f\lesssim g$ to denote there exists a positive constant $C$ such that $f\leq Cg $, and the notation $f\thickapprox g$ means that there exist positive constants $C_1, C_2$ such that $C_1 g\leq f\leq C_2g$. Throughout this paper, the letter $\vec{p}$ denotes $n$-tuples of the numbers in $(0,\infty],n\geq1$, $\vec{p}=(p_{1},p_{2}, \cdots,p_{n})$, $0<\vec{p}\leq\infty$ means $0<p_{i}\leq\infty$ for each $i$. Furthermore, for $1\leq\vec{p}\leq\infty$, let $$\frac{1}{\vec{p}}=(\frac{1}{p_{1}},\frac{1}{p_{2}},\cdots,\frac{1}{p_{n}}), \quad\vec{p}'=(p_{1}',p_{2}',\cdots,p_{n}'),$$
where $p_{i}'=\frac{p_{i}}{p_{i}-1} (i=1,2,\cdots,n)$ is the conjugate exponent of $p_i$.

\section{$\lambda$-central  Morrey spaces and $BMO$ spaces with mixed norm}
In this section, we first recall the definition of the mixed Lebesgue spaces  and then give the definitions of  $\lambda$-central Morrey spaces and $BMO$ spaces with mixed norm.

For $0<\vec{p}<\infty$, the mixed Lebesgue norm $\|\cdot\|_{L^{\vec{p}}(\mathbb{R}^{n})}$ is defined by
$$\|f\|_{L^{\vec{p}}(\mathbb{R}^{n})}:=\Bigg(\int_{\mathbb{R}}\cdots\bigg(\int_{\mathbb{R}}\big(\int_{\mathbb{R}}|f(x_{1},x_{2},\cdots,x_{n})|^{p_{1}}
dx_{1}\big)^{\frac{p_{2}}{p_{1}}}dx_{2}\bigg)^{\frac{p_{3}}{p_{2}}}\cdots dx_{n}\Bigg)^{\frac{1}{p_{n}}},$$
where $f:\mathbb{R}^{n}\rightarrow\mathbb{C}$ is a measurable function. If $p_{j}=\infty$, for some $j=1,\cdots,n$, we only need to make some suitable modifications.  We define the mixed Lebesgue space $L^{\vec{p}}(\mathbb{R}^{n})$ to be the set of all $f\in\mathscr{M}(\mathbb{R}^{n})$ with $\|f\|_{L^{\vec{p}}(\mathbb{R}^{n})}<\infty$. The mixed Lebesgue space $L^{\vec{p}}(\mathbb{R}^{n})$ reduces to the classical Lebesgue spaces
$L^{p}(\mathbb{R}^{n})$ when $p_{1}=p_{2}=\cdots=p_{n}=p$, and we refer the readers to \cite{BP} for more details.

Nagayama replaced the $L^{q}(\mathbb{R}^{n})$ norm by the mixed $L^{\vec{q}}(\mathbb{R}^{n})$ norm in the definition of classical Morrey space $M^{p}_{q}(\mathbb{R}^{n})$ in \cite{CM}, introduced  the mixed Morrey spaces $M^{p}_{\vec{q}}(\mathbb{R}^{n})$ in \cite{TN1}. Very recently, Wei \cite{WM1} introduced mixed central Morrey spaces $\mathcal{M}^{p}_{\vec{q}}(\mathbb{R}^{n})$ as follows.

\quad\hspace{-22pt}{\bf Definition 2.1}(\cite{WM1}) {For $\vec{q}=(q_{1},\cdots,q_{n})\in(0,\infty]^{n}$ and $p\in(0,\infty]$ satisfy
$$\sum\limits_{j=1}^{n}\frac{1}{q_{j}}\geq\frac{n}{p},$$
the mixed central Morrey space is defined by
$$\mathcal{M}^{p}_{\vec{q}}(\mathbb{R}^{n}):=\{f\in\mathscr{M}(\mathbb{R}^{n}):\|f\|_{\mathcal{M}^{p}_{\vec{q}}(\mathbb{R}^{n})}<\infty\},$$
where the mixed central Morrey quasi-norm
$$\|f\|_{\mathcal{M}^{p}_{\vec{q}}(\mathbb{R}^{n})}:=\sup\limits_{r>0}|B(0,r)|^{\frac{1}{p}-\frac{1}{n}\big(\sum^{n}_{j=1}\frac{1}{q_{j}}\big)}
\|f\chi_{\mathrm{B}(0,r)}\|_{L^{\vec{q}}(\mathbb{R}^{n})}.$$
}

In the study of the theory of commutators, there is the class of function space that has quite different from bounded mean oscillation space $BMO(\mathbb{R}^{n})$, that is,  central bounded mean oscillation space $CBMO^q(\mathbb{R}^{n})(1\leq q<\infty)$, which is introduced by Lu and Yang
in \cite{LY2}.
Now we recall the  definition of $CBMO^q(\mathbb{R}^{n})$  as follows.

For $1\leq q<\infty$, a function $f\in L^{q}_{loc}(\mathbb{R}^{n})$ is said to belong to $CBMO^q(\mathbb{R}^{n})$ if
$$\|f\|_{CBMO^q(\mathbb{R}^{n})}:=\sup\limits_{r>0}\Big(\frac{1}{|B(0,r)|}\int_{B(0,r)}|f(x)-f_{B(0,r)}|^qdx\Big)^{1/q}<\infty,$$
here and in the sequel, for $r>0$, $B(0,r)$ denote by the open ball centered at $0$ of radius $r$, $|B(0,r)|$ the Lebesgue measure of the ball $B(0,r)$ and
$$f_{B(0,r)}=\frac{1}{|B(0,r)|}\int_{B(0,r)}f(x)dx.$$
In fact, the space $CBMO^q(\mathbb{R}^{n})$ can be regarded as a local version of $BMO(\mathbb{R}^{n})$. However, their properties have quite different, for example, the famous John-Nirenberg inequality for $BMO(\mathbb{R}^{n})$ space no longer holds in the $CBMO^q(\mathbb{R}^{n})$ space. In addition,  $BMO(\mathbb{R}^{n})$ is strictly included in $\cap_{q>1}CBMO^q$.
\medskip

The corresponding mixed $CBMO^q(\mathbb{R}^{n})$ is given as follows.

\quad\hspace{-22pt}{\bf Definition 2.2}(\cite{WM1}) {Let $\vec{q}=(q_{1},\cdots,q_{n})\in[1,\infty)^{n}$, then the mixed central bounded mean oscillation space $CBMO^{\vec{q}}(\mathbb{R}^{n})$ is defined by
$$CBMO^{\vec{q}}(\mathbb{R}^{n}):=\{f\in L^{q}_{loc}(\mathbb{R}^{n}):\|f\|_{CBMO^{\vec{q}}(\mathbb{R}^{n})}<\infty\},$$
where
$$\|f\|_{CBMO^{\vec{q}}(\mathbb{R}^{n})}=\sup\limits_{r>0}\frac{\|(f-f_{B(0,r)})\chi_{B(0,r)}\|_{L^{\vec{q}}(\mathbb{R}^{n})}}
{\|\chi_{B(0,r)}\|_{L^{\vec{q}}(\mathbb{R}^{n})}}.$$
}

In addition, the classical $\lambda$-central $BMO(\mathbb{R}^{n})$ and $\lambda$-central Morrey spaces $\mathcal{B}^{q,\lambda}(\mathbb{R}^n)$, as generations of the spaces $BMO(\mathbb{R}^{n})$ and $M^{p}_{{q}}(\mathbb{R}^{n})$, are rather important function  spaces, which were introduced by Alvarez, Guzm\'{a}n-Partida and Lakey in \cite{JM}. It is based on these facts, we will further extend these function spaces to the case of the mixed norm.

\quad\hspace{-22pt}{\bf Definition 2.3} {Let $1<\vec{q}<\infty$ and $\lambda\in\mathbb{R}$,
the mixed $\lambda$-central Morrey space $\mathcal{B}^{\vec{q},\lambda}(\mathbb{R}^n)$ is defined by
$$\mathcal{B}^{\vec{q},\lambda}(\mathbb{R}^n):=\{f\in L^{\vec{q}}_{loc}(\mathbb{R}^{n}):\|f\|_{\mathcal{B}^{\vec{q},\lambda}(\mathbb{R}^n)}<\infty\},$$
where
$$\|f\|_{\mathcal{B}^{\vec{q},\lambda}(\mathbb{R}^n)}:=\sup\limits_{r>0}\frac{\|f\chi_{B(0,r)}\|_{L^{\vec{q}}(\mathbb{R}^{n})}}{|B(0,r)|^{\lambda}\|\chi_{B(0,r)}\|_{L^{\vec{q}}(\mathbb{R}^{n})}}.$$
}

\quad\hspace{-22pt}{\bf Remark 2.1} {If $\lambda=-\frac{1}{p}$, then the space $\mathcal{B}^{\vec{q},\lambda}(\mathbb{R}^n)$ reduces to Definition 2.1. The space $\mathcal{B}^{\vec{q},\lambda}(\mathbb{R}^n)$ is just the $\mathcal{B}^{q,\lambda}(\mathbb{R}^n)$ introduced in \cite{JM} when $\vec{q}=q$.
}

\quad\hspace{-22pt}{\bf Definition 2.4} {Let $1<\vec{q}<\infty$ and $\lambda<\frac{1}{n}$,
the mixed $\lambda$-central  bounded mean oscillation space $CBMO^{\vec{q},\lambda}(\mathbb{R}^{n})$ is defined by
$$CBMO^{\vec{q},\lambda}(\mathbb{R}^{n}):=\{f\in L^{\vec{q}}_{loc}(\mathbb{R}^{n}):\|f\|_{CBMO^{\vec{q},\lambda}(\mathbb{R}^{n})}<\infty\},$$
where
$$\|f\|_{CBMO^{\vec{q},\lambda}(\mathbb{R}^{n})}:=\sup\limits_{r>0}\frac{\|(f-f_{B(0,r)})\chi_{B(0,r)}\|_{L^{\vec{q}}(\mathbb{R}^{n})}}{|B(0,r)|^{\lambda}\|\chi_{B(0,r)}\|_{L^{\vec{q}}(\mathbb{R}^{n})}}.$$
}

\quad\hspace{-22pt}{\bf Remark 2.2} {If $\lambda=0$, then we can easily get the Definition 2.2. The space $CBMO^{\vec{q},\lambda}(\mathbb{R}^{n})$ is
just the $CBMO^{q,\lambda}(\mathbb{R}^{n})$ given in \cite{JM} when $\vec{q}=q$.
}

Next we give the relationship between the spaces $\mathcal{B}^{\vec{q},\lambda}(\mathbb{R}^n)$ and $CBMO^{\vec{q},\lambda}(\mathbb{R}^{n})$.

\begin{theorem} \label{th2.1}
Let $1<\vec{q}<\infty$, if $\lambda\in\mathbb{R}$, then $\mathcal{B}^{\vec{q},\lambda}(\mathbb{R}^n)$ is a Banach space; if $\lambda<\frac{1}{n}$, then $\mathcal{B}^{\vec{q},\lambda}(\mathbb{R}^n)$ is continuously included in $CBMO^{\vec{q},\lambda}(\mathbb{R}^{n})$.
\end{theorem}

\begin{proof}[Proof]
Firstly we will prove $\mathcal{B}^{\vec{q},\lambda}(\mathbb{R}^n)$ is a Banach space. Let $\{f_{k}\}$ be a Cauchy sequence in $\mathcal{B}^{\vec{q},\lambda}(\mathbb{R}^n)$. Then for $\epsilon>0$, there exists $N(\epsilon)\in \mathbb{N}$ such that $\|f_{i}-f_{j}\|_{\mathcal{B}^{\vec{q},\lambda}(\mathbb{R}^n)}<\epsilon$ for $i,j>N(\epsilon)$. Let $n_{k}=N(\frac{1}{2^{k}})$, then there is a subsequence, denoted by $\{f_{n_{k}}\}$, such that
$$\|f_{n_{k+1}}-f_{n_{k}}\|_{\mathcal{B}^{\vec{q},\lambda}(\mathbb{R}^n)}<\frac{1}{2^{k}}.$$
We set
$$f(x)=f_{n_{1}}(x)+\sum\limits_{k=1}^{\infty}\big(f_{n_{k+1}}(x)-f_{n_{k}}(x)\big), \quad\quad x\in\mathbb{R}^{n},$$
and
$$g(x)=|f_{n_{1}}(x)|+\sum\limits_{k=1}^{\infty}\big|f_{n_{k+1}}(x)-f_{n_{k}}(x)\big|, \quad\quad x\in\mathbb{R}^{n}.$$
Let $S_{N}(f)=f_{n_{1}}+\sum_{k=1}^{N-1}(f_{n_{k+1}}-f_{n_{k}})=f_{n_{N}}$ and $S_{N}(g)=|f_{n_{1}}|+\sum_{k=1}^{N-1}|f_{n_{k+1}}-f_{n_{k}}|.$
Then by the Minkowski inequality, we have
\begin{align*}
\quad\quad&\|S_{N}(g)\|_{\mathcal{B}^{\vec{q},\lambda}(\mathbb{R}^n)}=
\sup\limits_{r>0}\frac{\|S_{N}(g)\chi_{B(0,r)}\|_{L^{\vec{q}}(\mathbb{R}^{n})}}{|B(0,r)|^{\lambda}\|\chi_{B(0,r)}\|_{L^{\vec{q}}(\mathbb{R}^{n})}}\\
&\quad\lesssim\sup\limits_{r>0}\frac{\|f_{n_{1}}\chi_{B(0,r)}\|_{L^{\vec{q}}(\mathbb{R}^{n})}}{|B(0,r)|^{\lambda}\|\chi_{B(0,r)}\|_{L^{\vec{q}}(\mathbb{R}^{n})}}+
\sum\limits_{k=1}^{N-1}\sup\limits_{r>0}\frac{\|(f_{n_{k+1}}-f_{n_{k}})\chi_{B(0,r)}\|_{L^{\vec{q}}(\mathbb{R}^{n})}}{|B(0,r)|^{\lambda}\|\chi_{B(0,r)}\|_{L^{\vec{q}}(\mathbb{R}^{n})}}\\
&\quad\leq\|f_{n_{1}}\|_{\mathcal{B}^{\vec{q},\lambda}(\mathbb{R}^n)}+\sum\limits_{k=1}^{N-1}\frac{1}{2^{k}}\lesssim
\|f_{n_{1}}\|_{\mathcal{B}^{\vec{q},\lambda}(\mathbb{R}^n)}+1.
\end{align*}
So we have $\|g\|_{\mathcal{B}^{\vec{q},\lambda}(\mathbb{R}^n)}<\infty$. Noting that $|f|\leq g$, then $\|f\|_{\mathcal{B}^{\vec{q},\lambda}(\mathbb{R}^n)}\leq\|g\|_{\mathcal{B}^{\vec{q},\lambda}(\mathbb{R}^n)}<\infty$. Thus we have $f\in\mathcal{B}^{\vec{q},\lambda}(\mathbb{R}^n)$.
Since
\begin{align*}
\lim\limits_{N\rightarrow\infty}\|f-f_{n_{N}}\|_{\mathcal{B}^{\vec{q},\lambda}(\mathbb{R}^n)}&\leq\lim\limits_{N\rightarrow\infty}
\sum\limits_{k=N}^{\infty}\|f_{n_{k+1}}-f_{n_{k}}\|_{\mathcal{B}^{\vec{q},\lambda}(\mathbb{R}^n)}\\
&\leq\lim\limits_{N\rightarrow\infty}\sum\limits_{k=N}^{\infty}2^{-k}=\lim\limits_{N\rightarrow\infty}2^{-N+1}=0.
\end{align*}
From this we know that the sequence $\{f_{n_{k}}\}$ converges to $f$ in $\mathcal{B}^{\vec{q},\lambda}(\mathbb{R}^n)$, then the Cauchy
sequence $\{f_{k}\}$ converges to $f$ in $\mathcal{B}^{\vec{q},\lambda}(\mathbb{R}^n)$. Thus $\mathcal{B}^{\vec{q},\lambda}(\mathbb{R}^n)$ is a Banach space.

By the H\"{o}lder inequality we have
\begin{align*}
&\|f\|_{CBMO^{\vec{q},\lambda}(\mathbb{R}^{n})}=
\sup\limits_{r>0}\frac{\|(f-f_{B(0,r)})\chi_{B(0,r)}\|_{L^{\vec{q}}(\mathbb{R}^{n})}}{|B(0,r)|^{\lambda}\|\chi_{B(0,r)}\|_{L^{\vec{q}}(\mathbb{R}^{n})}}\\
&\quad\lesssim\Big(\sup\limits_{r>0}\frac{\|f\chi_{B(0,r)}\|_{L^{\vec{q}}(\mathbb{R}^{n})}}{|B(0,r)|^{\lambda}\|\chi_{B(0,r)}\|_{L^{\vec{q}}(\mathbb{R}^{n})}}
+\sup\limits_{r>0}\frac{\|f_{B(0,r)}\chi_{B(0,r)}\|_{L^{\vec{q}}(\mathbb{R}^{n})}}{|B(0,r)|^{\lambda}\|\chi_{B(0,r)}\|_{L^{\vec{q}}(\mathbb{R}^{n})}}\Big)\\
&\quad\leq\Big(\sup\limits_{r>0}\frac{\|f\chi_{B(0,r)}\|_{L^{\vec{q}}(\mathbb{R}^{n})}}{|B(0,r)|^{\lambda}\|\chi_{B(0,r)}\|_{L^{\vec{q}}(\mathbb{R}^{n})}}
+\sup\limits_{r>0}\frac{\|\chi_{B(0,r)}\|_{L^{\vec{q}}(\mathbb{R}^{n})}\cdot|f_{B(0,r)}|}{|B(0,r)|^{\lambda}\|\chi_{B(0,r)}\|_{L^{\vec{q}}(\mathbb{R}^{n})}}\Big)\\
&\quad\lesssim\Big(\sup\limits_{r>0}\frac{\|f\chi_{B(0,r)}\|_{L^{\vec{q}}(\mathbb{R}^{n})}}{|B(0,r)|^{\lambda}\|\chi_{B(0,r)}\|_{L^{\vec{q}}(\mathbb{R}^{n})}}\\
&\quad\quad+\sup\limits_{r>0}
\frac{\|\chi_{B(0,r)}\|_{L^{\vec{q}}(\mathbb{R}^{n})}\cdot\big(\frac{1}{B(0,r)}\|f\chi_{B(0,r)}\|_{L^{\vec{q}}(\mathbb{R}^{n})}\cdot
\|\chi_{B(0,r)}\|_{L^{\vec{q}'}(\mathbb{R}^{n})}\big)}{|B(0,r)|^{\lambda}\|\chi_{B(0,r)}\|_{L^{\vec{q}}(\mathbb{R}^{n})}}\Big)\\
&\quad\lesssim\Big(\sup\limits_{r>0}\frac{\|f\chi_{B(0,r)}\|_{L^{\vec{q}}(\mathbb{R}^{n})}}{|B(0,r)|^{\lambda}\|\chi_{B(0,r)}\|_{L^{\vec{q}}(\mathbb{R}^{n})}}\Big)\\
&\quad\lesssim\|f\|_{\mathcal{B}^{\vec{q},\lambda}(\mathbb{R}^n)},
\end{align*}
so we obtain $\mathcal{B}^{\vec{q},\lambda}(\mathbb{R}^n)$ is continuously included in $CBMO^{\vec{q},\lambda}(\mathbb{R}^{n})$.
Thus we complete the proof of Theorems \ref{th2.1}.
\end{proof}

In the application of the space $CBMO^{\vec{q},\lambda}(\mathbb{R}^{n})$, we often use the following significant property.

\begin{proposition}\label{pro2.1} Suppose that $f\in CBMO^{\vec{q},\lambda}(\mathbb{R}^{n}), 1<\vec{q}<\infty, \lambda<\frac{1}{n}$ and $r_{1},r_{2}\in\mathbb{R}^+$. Then
$$\frac{\|(b-b_{B(0,r_{2})})\chi_{B(0,r_{1})}\|_{L^{\vec{q}}(\mathbb{R}^{n})}}{|B(0,r_{1})|^{\lambda}\|\chi_{B(0,r_{1})}\|_{L^{\vec{q}}(\mathbb{R}^{n})}}\lesssim \Big(1+\Big|\ln\frac{r_{1}}{r_{2}}\Big|\Big)\|f\|_{CBMO^{\vec{q},\lambda}(\mathbb{R}^{n})}.$$
\end{proposition}
The proof of Proposition \ref{pro2.1} is a slight modification of Lemma 2.5 in \cite{YT} and we omit the details here.

\section{Boundedness of the fractional integral operators}
The standard Calder\'{o}n-Zygmund operator $T$ is defined by
$$Tf(x)=p.v.\int_{\mathbb{R}^{n}}K(x-y)f(y)dy,$$
with the kernel $K$ satisfying the following size condition:
$$|K(x)|\leq C|x|^{-n}, x\neq0,$$
and some smoothness assumption. In fact, the Calder\'{o}n-Zygmund operator is a direct generalization of the Hilbert transform and the Riesz transform. The former is originated from researches of boundary value of conjugate harmonic functions on the upper half-plane, and the latter is tightly associated to the regularity of solution of second order elliptic equation, for more information, see the monograph \cite{SE}.

For $0<\alpha<n$, the fractional integral operator $T_{\alpha}$ is defined by
$$T_{\alpha}f(x)=\int_{\mathbb{R}^{n}}\frac{f(y)}{|x-y|^{n-\alpha}}dy.$$
It is closely related to the Laplacian operator of fractional degree. When $n>2$ and $\alpha=2, T_{\alpha}f$ is a solution of Poisson equation $-\triangle u=f$. The importance of fractional integral operators is owing to the fact that they are smooth operators and have been
extensively used in various areas such as potential analysis, harmonic analysis and partial differential equations.

In this section, we are devoted to establishing the boundedness of the operators $T$ and $T_\alpha$  on the mixed $\lambda$-central  Morrey spaces, respectively. However, the core techniques used in the proofs are similar, or one can repeats the ideas of the fractional integral operator $T_\alpha$ verbatim to prove the boundedness of the operators $T$.  Therefore, to unify the notation, we denote by $T_{\alpha} (0\leq\alpha<n)$ the fractional integral operator when $0<\alpha<n$, and the standard Calder\'{o}n-Zygmund operator when $\alpha=0$ (i.e. $T=T_{0})$.

For $0\leq\alpha<n$, let $b$ be a locally integrable function, the commutator $[b,T_{\alpha}]$ generalized by the operator $T_{\alpha}$ and the symbol $b$ is defined by
    $$[b,T_{\alpha}]f:=bT_{\alpha}f-T_{\alpha}(bf).$$

In \cite{BP}, Benedek and Panzone established the boundedness of the fractional integral operators on the mixed Lebesgue spaces. In 2019, Nogayama \cite{TN1} proved  the boundedness of the fractional integral operators  and the singular integral operators on the mixed Morrey spaces. Meanwhile, the boundedness of the commutator of the fractional integral operators was also obtained in \cite{TN2}. Recently, Wei \cite{WM1} extended the results to the mixed central Morrey spaces.

 Next we will prove the boundedness of the operator $T_\alpha$ and its commutator  $[b,T_{\alpha}]$ on the mixed $\lambda$-central  Morrey spaces.

\begin{theorem} \label{th3.1}
Let $0\leq\alpha<n,1<\vec{p},\vec{q}<\infty$ satisfy conditions $\vec{p}<\frac{n}{\alpha}$ and $\alpha=\sum^{n}_{i=1}\frac{1}{p_{i}}-\sum^{n}_{i=1}\frac{1}{q_{i}}$. If $\lambda_{1}<-\frac{\alpha}{n}$ and $\lambda=\lambda_{1}+\frac{\alpha}{n}$,
then
$$\|T_{\alpha}f\|_{\mathcal{B}^{\vec{q},\lambda}(\mathbb{R}^n)}\lesssim\|f\|_{\mathcal{B}^{\vec{p},\lambda_1}(\mathbb{R}^n)}.$$
\end{theorem}

\begin{proof}[Proof]
Let $f$ be a function in $\mathcal{B}^{\vec{p},\lambda_1}(\mathbb{R}^n)$. For any fixed $r>0$, denote $B(0,r)$ by $B$. It suffices to show that the fact
$$\|T_{\alpha}f\chi_{B(0,r)}\|_{L^{\vec{q}}(\mathbb{R}^n)}\lesssim|B|^{\lambda}\|\chi_{B(0,r)}\|_{L^{\vec{q}}(\mathbb{R}^n)}
\|f\|_{\mathcal{B}^{\vec{p},\lambda_1}(\mathbb{R}^n)}$$
holds.

For any ball $B$, let $2B=B(0,2r)$. We write $f$ as $f=f_{1}+f_{2}$, where $f_{1}=f\chi_{2B}$ and $f_{2}=f\chi_{(2B)^{c}}$, via the Minkowski inequality, we have
\begin{align*}
\|T_{\alpha}f\chi_{B(0,r)}\|_{L^{\vec{q}}(\mathbb{R}^n)}&\leq\|T_{\alpha}f_{1}\chi_{B(0,r)}\|_{L^{\vec{q}}(\mathbb{R}^n)}
+\|T_{\alpha}f_{2}\chi_{B(0,r)}\|_{L^{\vec{q}}(\mathbb{R}^n)}\\
&=:I_{1}+I_{2}.
\end{align*}
We first estimate $I_{1}$. Note that $f_{1}\in L^{\vec{q}}(\mathbb{R}^n)$ and $T_{\alpha}$ is bounded from $L^{\vec{p}}(\mathbb{R}^n)$ to $L^{\vec{q}}(\mathbb{R}^n)$, we get
\begin{align*}
\|T_{\alpha}f_{1}\chi_{B(0,r)}\|_{L^{\vec{q}}(\mathbb{R}^n)}&\leq\|T_{\alpha}f_{1}\|_{L^{\vec{q}}(\mathbb{R}^n)}
\lesssim\|f_{1}\|_{L^{\vec{p}}(\mathbb{R}^n)}\\
&=\|f\chi_{B(0,2r)}\|_{L^{\vec{p}}(\mathbb{R}^n)}.
\end{align*}
Thus, the condition $\lambda=\lambda_{1}+\frac{\alpha}{n}$ deduces that
\begin{align*}
&\|T_{\alpha}f_{1}\chi_{B(0,r)}\|_{L^{\vec{q}}(\mathbb{R}^n)}\\
&\lesssim|B(0,2r)|^{\lambda_{1}}\|\chi_{B(0,2r)}\|_{L^{\vec{p}}(\mathbb{R}^n)}
\|f\|_{\mathcal{B}^{\vec{p},\lambda_1}(\mathbb{R}^n)}\\
&=|B|^{\lambda}\|\chi_{B(0,r)}\|_{L^{\vec{q}}(\mathbb{R}^n)}\frac{|B(0,2r)|^{\lambda_{1}}}{|B|^{\lambda}}
\frac{\|\chi_{B(0,2r)}\|_{L^{\vec{p}}(\mathbb{R}^n)}}{\|\chi_{B(0,r)}\|_{L^{\vec{q}}(\mathbb{R}^n)}}\|f\|_{\mathcal{B}^{\vec{p},\lambda_1}(\mathbb{R}^n)}\\
&\lesssim|B|^{\lambda}\|\chi_{B(0,r)}\|_{L^{\vec{q}}(\mathbb{R}^n)}\|f\|_{\mathcal{B}^{\vec{p},\lambda_1}(\mathbb{R}^n)}.
\end{align*}

Now we turn to prove  $I_{2}$. For $x\in B$ and $y\in2^{k+1}B\backslash2^{k}B$, we have $|x-y|\lesssim|2^{k}B|^{1/n}$, together with the H\"{o}lder inequality, we obtain

\begin{align*}
|T_{\alpha}f_{2}(x)|&=|T_{\alpha}(f\chi_{(2B)^{c}})(x)|\\
&\leq\sum\limits_{k=1}^{\infty}\int_{2^{k+1}B\backslash2^{k}B}\frac{|f(y)|}{|x-y|^{n-\alpha}}dy\\
&\lesssim\sum\limits_{k=1}^{\infty}|2^{k}B|^{-1+\frac{\alpha}{n}}\int_{2^{k+1}B}|f(y)|dy\\
&\lesssim\sum\limits_{k=1}^{\infty}|2^{k}B|^{-1+\frac{\alpha}{n}}\|\chi_{2^{k+1}B}\|_{L^{\vec{p}'}(\mathbb{R}^n)}
\|f\chi_{2^{k+1}B}\|_{L^{\vec{p}}(\mathbb{R}^n)}.
\end{align*}
Therefore,
\begin{align*}
&\|T_{\alpha}f_{2}\chi_{B(0,r)}\|_{L^{\vec{q}}(\mathbb{R}^n)}\\
&\lesssim\sum\limits_{k=1}^{\infty}|2^{k}B|^{-1+\frac{\alpha}{n}}
\|\chi_{2^{k+1}B}\|_{L^{\vec{p}'}(\mathbb{R}^n)}\|f\chi_{2^{k+1}B}\|_{L^{\vec{p}}(\mathbb{R}^n)}\|\chi_{B}\|_{L^{\vec{q}}(\mathbb{R}^n)}\\
&=\sum\limits_{k=1}^{\infty}|2^{k}B|^{-1+\frac{\alpha}{n}}|2^{k+1}B|^{\lambda_{1}}\|\chi_{2^{k+1}B}\|_{L^{\vec{p}'}(\mathbb{R}^n)}
\|\chi_{B}\|_{L^{\vec{q}}(\mathbb{R}^n)}\|\chi_{2^{k+1}B}\|_{L^{\vec{p}}(\mathbb{R}^n)}\|f\|_{\mathcal{B}^{\vec{p},\lambda_1}(\mathbb{R}^n)}\\
&\lesssim\sum\limits_{k=1}^{\infty}|2^{k}B|^{-1+\frac{\alpha}{n}}|2^{k+1}B|^{\lambda_{1}}|2^{k+1}B|\|\chi_{B}\|_{L^{\vec{q}}(\mathbb{R}^n)}
\|f\|_{\mathcal{B}^{\vec{p},\lambda_1}(\mathbb{R}^n)}\\
&=\sum\limits_{k=1}^{\infty}2^{kn(-1+\frac{\alpha}{n})}|B|^{(-1+\frac{\alpha}{n})}2^{kn\lambda_{1}}|B|^{\lambda_{1}}2^{kn}|B|
\|\chi_{B}\|_{L^{\vec{q}}(\mathbb{R}^n)}\|f\|_{\mathcal{B}^{\vec{p},\lambda_1}(\mathbb{R}^n)}\\
&=|B|^{\lambda}\|\chi_{B}\|_{L^{\vec{q}}(\mathbb{R}^n)}\|f\|_{\mathcal{B}^{\vec{p},\lambda_1}(\mathbb{R}^n)}
\sum\limits_{k=1}^{\infty}2^{kn(\frac{\alpha}{n}+\lambda_{1})}\\
&\lesssim|B|^{\lambda}\|\chi_{B}\|_{L^{\vec{q}}(\mathbb{R}^n)}\|f\|_{\mathcal{B}^{\vec{p},\lambda_1}(\mathbb{R}^n)},
\end{align*}
where we used the assumption $\lambda_{1}<-\frac{\alpha}{n}$ in the last inequality. The proof of Theorem \ref{th3.1} is completed.
\end{proof}

\begin{theorem} \label{th3.2}
Let $0\leq\alpha<n, 1<\vec{p}_{1}, \vec{p_{2}}, \vec{q}<\infty$ satisfy conditions $\vec{p_{1}}<\frac{n}{\alpha}, {\vec{p}_{1}}'<\vec{p}_{2}$ and $\alpha=\sum^{n}_{i=1}\frac{1}{p_{1i}}+\sum^{n}_{i=1}\frac{1}{p_{2i}}-\sum^{n}_{i=1}\frac{1}{q_{i}}$. Let $0\leq\lambda_{2}<\frac{1}{n},\lambda_{1}<-\lambda_{2}-\frac{\alpha}{n}$ and $\lambda=\lambda_{1}+\lambda_{2}+\frac{\alpha}{n}$.
If $b\in CBMO^{\vec{p}_{2},\lambda_{2}}(\mathbb{R}^{n})$,
then
$$\|[b,T_{\alpha}]f\|_{\mathcal{B}^{\vec{q},\lambda}(\mathbb{R}^n)}\lesssim\|b\|_{{CBMO}^{\vec{p}_{2},\lambda_{2}}(\mathbb{R}^n)}
\|f\|_{\mathcal{B}^{\vec{p}_{1},\lambda_{1}}(\mathbb{R}^n)}.$$
\end{theorem}

\begin{proof}[Proof]
Let $f$ be a function in $\mathcal{B}^{\vec{p},\lambda_1}(\mathbb{R}^n)$. For any fixed $r>0$, denote $B(0,r)$ by $B$. We need to check that the fact
$$\|[b,T_{\alpha}]f\chi_{B}\|_{L^{\vec{q}}(\mathbb{R}^n)}\lesssim|B|^{\lambda}\|\chi_{B}\|_{L^{\vec{q}}(\mathbb{R}^n)}
\|b\|_{{CBMO}^{\vec{p}_{2},\lambda_{2}}(\mathbb{R}^n)}\|f\|_{\mathcal{B}^{\vec{p}_{1},\lambda_{1}}(\mathbb{R}^n)}$$
holds.
The Minkowski inequality gives that
\begin{align*}
&\|[b,T_{\alpha}]f\chi_{B}\|_{L^{\vec{q}}(\mathbb{R}^n)}\\
&\leq\|(b(x)-b_B)T_{\alpha}(f_1)\chi_{B}\|_{L^{\vec{q}}(\mathbb{R}^n)}+\|T_{\alpha}((b-b_B)f_1)\chi_{B}\|_{L^{\vec{q}}(\mathbb{R}^n)}\\
&\quad+\|(b(x)-b_B)T_{\alpha}(f_2)\chi_{B}\|_{L^{\vec{q}}(\mathbb{R}^n)}+\|T_{\alpha}((b-b_B)f_2)\chi_{B}\|_{L^{\vec{q}}(\mathbb{R}^n)}\\
&=:J_1+J_2+J_3+J_4.
\end{align*}

 We first estimate $J_1$. We may choose $\vec{t}$ such that $\frac{1}{\vec{q}}=\frac{1}{\vec{p_{2}}}+\frac{1}{\vec{t}}$, then $\sum^{n}_{i=1}\frac{1}{t_{i}}=\sum^{n}_{i=1}\frac{1}{p_{1i}}-\alpha$. Noticing that $\vec{p_{1}}<\frac{n}{\alpha}$, by H\"{o}lder's inequality and the $(L^{\vec{p_{1}}}(\mathbb{R}^{n}),L^{\vec{t}}(\mathbb{R}^{n}))$-boundedness of $T_{\alpha}$, we get
\begin{align*}
J_1&=\|(b(x)-b_B)T_{\alpha}(f_1)\chi_{B}\|_{L^{\vec{q}}(\mathbb{R}^n)}\\
&\leq\|(b(x)-b_B)\chi_{B}\|_{L^{\vec{p}_{2}}(\mathbb{R}^n)}\|T_{\alpha}(f_1)\|_{L^{\vec{t}}(\mathbb{R}^n)}\\
&\lesssim \|b\|_{{CBMO}^{\vec{p}_{2},\lambda_{2}}(\mathbb{R}^n)}|B|^{\lambda_{2}}\|\chi_{B}\|_{L^{\vec{p}_{2}}(\mathbb{R}^n)}
\|f\chi_{2B}\|_{L^{\vec{p}_{1}}(\mathbb{R}^n)}\\
&\lesssim \|b\|_{{CBMO}^{\vec{p}_{2},\lambda_{2}}(\mathbb{R}^n)}|B|^{\lambda_{2}}\|\chi_{B}\|_{L^{\vec{p}_{2}}(\mathbb{R}^n)}
\|f\|_{\mathcal{B}^{\vec{p}_{1},\lambda_{1}}(\mathbb{R}^n)}|2B|^{\lambda_{1}}\|\chi_{2B}\|_{L^{\vec{p}_{1}}(\mathbb{R}^n)}\\
&\lesssim \|b\|_{{CBMO}^{\vec{p}_{2},\lambda_{2}}(\mathbb{R}^n)}\|f\|_{\mathcal{B}^{\vec{p}_{1},\lambda_{1}}(\mathbb{R}^n)}
|B|^{\lambda_{1}+\lambda_{2}}\|\chi_{B}\|_{L^{\vec{p}_{2}}(\mathbb{R}^n)}\|\chi_{B}\|_{L^{\vec{p}_{1}}(\mathbb{R}^n)}\\
&\lesssim|B|^{\lambda}\|\chi_{B}\|_{L^{\vec{q}}(\mathbb{R}^n)}\|b\|_{{CBMO}^{\vec{p}_{2},\lambda_{2}}(\mathbb{R}^n)}
\|f\|_{\mathcal{B}^{\vec{p}_{1},\lambda_{1}}(\mathbb{R}^n)}.
\end{align*}
Next we estimate $J_2$. Choose $\vec{l}$ such that $\frac{1}{\vec{l}}=\frac{1}{\vec{p_{1}}}+\frac{1}{\vec{p_{2}}}$, then $\alpha=\sum^{n}_{i=1}\frac{1}{l_{i}}-\sum^{n}_{i=1}\frac{1}{q_{i}}$. Since $1<\vec{p_{1}}<\frac{n}{\alpha}$ and $\vec{p_{1}}'<\vec{p_{2}}<\infty$, we obtain $1<\vec{l}<\frac{n}{\alpha}$. Using H\"{o}lder's inequality and the
$(L^{\vec{l}}(\mathbb{R}^{n}),L^{\vec{q}}(\mathbb{R}^{n}))$-boundedness of $T_{\alpha}$, we have
\begin{align*}
J_2&=\|T_{\alpha}((b-b_B)f_1)\chi_{B}\|_{L^{\vec{q}}(\mathbb{R}^n)}\\
&\lesssim\|(b(x)-b_B)(f\chi_{2B})\|_{L^{\vec{l}}(\mathbb{R}^n)}\\
&\lesssim\|(b(x)-b_B)\chi_{2B}\|_{L^{\vec{p}_{2}}(\mathbb{R}^n)}\|f\chi_{2B}\|_{L^{\vec{p}_{1}}(\mathbb{R}^n)}\\
&\lesssim\big(\|(b(x)-b_{2B})\chi_{2B}\|_{L^{\vec{p}_{2}}(\mathbb{R}^n)}+|b_{2B}-b_B|\|\chi_{2B}\|_{L^{\vec{p}_{2}}(\mathbb{R}^n)}\big)
\|f\chi_{2B}\|_{L^{\vec{p}_{1}}(\mathbb{R}^n)}\\
&\lesssim\|(b(x)-b_{2B})\chi_{2B}\|_{L^{\vec{p}_{2}}(\mathbb{R}^n)}\|f\chi_{2B}\|_{L^{\vec{p}_{1}}(\mathbb{R}^n)}\\
&\lesssim\|b\|_{{CBMO}^{\vec{p}_{2},\lambda_{2}}(\mathbb{R}^n)}|2B|^{\lambda_{2}}\|\chi_{2B}\|_{L^{\vec{p}_{2}}(\mathbb{R}^n)}
\|f\chi_{2B}\|_{L^{\vec{p}_{1}}(\mathbb{R}^n)}\\
&\lesssim \|b\|_{{CBMO}^{\vec{p}_{2},\lambda_{2}}(\mathbb{R}^n)}|2B|^{\lambda_{2}}\|\chi_{2B}\|_{L^{\vec{p}_{2}}(\mathbb{R}^n)}
\|f\|_{\mathcal{B}^{\vec{p}_{1},\lambda_{1}}(\mathbb{R}^n)}|2B|^{\lambda_{1}}\|\chi_{2B}\|_{L^{\vec{p}_{1}}(\mathbb{R}^n)}\\
&\lesssim \|b\|_{{CBMO}^{\vec{p}_{2},\lambda_{2}}(\mathbb{R}^n)}\|f\|_{\mathcal{B}^{\vec{p}_{1},\lambda_{1}}(\mathbb{R}^n)}
|B|^{\lambda_{1}+\lambda_{2}}\|\chi_{B}\|_{L^{\vec{p}_{2}}(\mathbb{R}^n)}\|\chi_{B}\|_{L^{\vec{p}_{1}}(\mathbb{R}^n)}\\
&\leq|B|^{\lambda}\|\chi_{B}\|_{L^{\vec{q}}(\mathbb{R}^n)}\|b\|_{{CBMO}^{\vec{p}_{2},\lambda_{2}}(\mathbb{R}^n)}
\|f\|_{\mathcal{B}^{\vec{p}_{1},\lambda_{1}}(\mathbb{R}^n)}.
\end{align*}

For $J_3$, it follows from $\lambda_{1}<-\lambda_{2}-\frac{\alpha}{n}\leq-\frac{\alpha}{n}$ that
\begin{align*}
&|T_{\alpha}(f\chi_{(2B)^{c}})(x)|\\
&\lesssim\sum\limits_{k=1}^{\infty}|2^{k}B|^{-1+\frac{\alpha}{n}}\|\chi_{2^{k+1}B}\|_{L^{\vec{p_{1}}'}(\mathbb{R}^n)}
\|f\chi_{2^{k+1}B}\|_{L^{\vec{p_{1}}}(\mathbb{R}^n)}\\
\end{align*}
\begin{align*}
&=\sum\limits_{k=1}^{\infty}|2^{k}B|^{-1+\frac{\alpha}{n}}\|\chi_{2^{k+1}B}\|_{L^{\vec{p_{1}}'}(\mathbb{R}^n)}
\|\chi_{2^{k+1}B}\|_{L^{\vec{p_{1}}}(\mathbb{R}^n)}|2^{k+1}B|^{\lambda_{1}}\|f\|_{\mathcal{B}^{\vec{p}_{1},\lambda_{1}}(\mathbb{R}^n)}\\
&\lesssim\sum\limits_{k=1}^{\infty}|2^{k}B|^{\lambda_{1}+\frac{\alpha}{n}}\|f\|_{\mathcal{B}^{\vec{p}_{1},\lambda_{1}}(\mathbb{R}^n)}\\
&\lesssim|B|^{\lambda_{1}+\frac{\alpha}{n}}\|f\|_{\mathcal{B}^{\vec{p}_{1},\lambda_{1}}(\mathbb{R}^n)}.
\end{align*}

Set $\frac{1}{\vec{q}}=\frac{1}{\vec{p_{2}}}+\frac{1}{\vec{t}}$, using H\"{o}lder's inequality, we derive
\begin{align*}
J_3&=\|(b(x)-b_B)T_{\alpha}(f_2)\chi_{B}\|_{L^{\vec{q}}(\mathbb{R}^n)}\\
&=\|(b(x)-b_B)T_{\alpha}(f\chi_{(2B)^{c}})\chi_{B}\|_{L^{\vec{q}}(\mathbb{R}^n)}\\
&\lesssim|B|^{\lambda_{1}+\frac{\alpha}{n}}\|f\|_{\mathcal{B}^{\vec{p}_{1},\lambda_{1}}(\mathbb{R}^n)}\|(b(x)-b_B)\chi_{B}\|_{L^{\vec{q}}(\mathbb{R}^n)}\\
&\lesssim|B|^{\lambda_{1}+\frac{\alpha}{n}}\|f\|_{\mathcal{B}^{\vec{p}_{1},\lambda_{1}}(\mathbb{R}^n)}
\|(b(x)-b_B)\chi_{B}\|_{L^{\vec{p_{2}}}(\mathbb{R}^n)}\|\chi_{B}\|_{L^{\vec{t}}(\mathbb{R}^n)}\\
&\lesssim|B|^{\lambda_{1}+\lambda_{2}+\frac{\alpha}{n}}\|b\|_{{CBMO}^{\vec{p}_{2},\lambda_{2}}(\mathbb{R}^n)}
\|f\|_{\mathcal{B}^{\vec{p}_{1},\lambda_{1}}(\mathbb{R}^n)}\|\chi_{B}\|_{L^{\vec{p}_{2}}(\mathbb{R}^n)}\|\chi_{B}\|_{L^{\vec{t}}(\mathbb{R}^n)}\\
&\lesssim|B|^{\lambda}\|\chi_{B}\|_{L^{\vec{q}}(\mathbb{R}^n)}\|b\|_{{CBMO}^{\vec{p}_{2},\lambda_{2}}(\mathbb{R}^n)}
\|f\|_{\mathcal{B}^{\vec{p}_{1},\lambda_{1}}(\mathbb{R}^n)}.
\end{align*}

To estimate $J_4$, take $\vec{s}$ such that $\frac{1}{\vec{p_{1}}'}=\frac{1}{\vec{p_{2}}}+\frac{1}{\vec{s}}$, noticing that $\lambda_{2}\geq0$
and $\lambda_{1}<-\lambda_{2}-\frac{\alpha}{n}$, by H\"{o}lder's inequality, we obtain that
\begin{align*}
&|T_{\alpha}((b-b_B)f\chi_{(2B)^{c}})(x)|\\
&\leq\sum\limits_{k=1}^{\infty}\int_{2^{k+1}B\backslash2^{k}B}\frac{|b(y)-b_{B}||f(y)|}{|x-y|^{n-\alpha}}dy\\
&\lesssim\sum\limits_{k=1}^{\infty}|2^{k}B|^{-1+\frac{\alpha}{n}}\int_{2^{k+1}B}|b(y)-b_{B}||f(y)|dy\\
&\lesssim\sum\limits_{k=1}^{\infty}|2^{k}B|^{-1+\frac{\alpha}{n}}\|(b-b_{B})\chi_{2^{k+1}B}\|_{L^{\vec{p_{1}}'}(\mathbb{R}^n)}
\|f\chi_{2^{k+1}B}\|_{L^{\vec{p_{1}}}(\mathbb{R}^n)}\\
&\lesssim\sum\limits_{k=1}^{\infty}|2^{k}B|^{-1+\frac{\alpha}{n}}\|(b-b_{B})\chi_{2^{k+1}B}\|_{L^{\vec{p_{2}}}(\mathbb{R}^n)}
\|\chi_{2^{k+1}B}\|_{L^{\vec{s}}(\mathbb{R}^n)}\\
&\quad\times\|f\|_{\mathcal{B}^{\vec{p}_{1},\lambda_{1}}(\mathbb{R}^n)}|2^{k+1}B|^{\lambda_{1}}
\|\chi_{2^{k+1}B}\|_{L^{\vec{p_{1}}}(\mathbb{R}^n)}\\
&\lesssim\sum\limits_{k=1}^{\infty}|2^{k}B|^{\lambda_{1}-1+\frac{\alpha}{n}}\big(\|(b(x)-b_{2^{k+1}B})\chi_{2^{k+1}B}\|_{L^{\vec{p}_{2}}(\mathbb{R}^n)}\\
&\quad+|b_{2^{k+1}B}-b_B|\|\chi_{2^{k+1}B}\|_{L^{\vec{p}_{2}}(\mathbb{R}^n)}\big)\\
&\quad\quad\times\|f\|_{\mathcal{B}^{\vec{p}_{1},\lambda_{1}}(\mathbb{R}^n)}\|\chi_{2^{k+1}B}\|_{L^{\vec{s}}(\mathbb{R}^n)}
\|\chi_{2^{k+1}B}\|_{L^{\vec{p_{1}}}(\mathbb{R}^n)}\\
&\lesssim\sum\limits_{k=1}^{\infty}k|2^{k}B|^{\lambda_{1}-1+\frac{\alpha}{n}}\|b\|_{{CBMO}^{\vec{p}_{2},\lambda_{2}}(\mathbb{R}^n)}
|2^{k+1}B|^{\lambda_{2}}\|\chi_{2^{k+1}B}\|_{L^{\vec{p}_{2}}(\mathbb{R}^n)}
\end{align*}
\begin{align*}
&\quad\times\|f\|_{\mathcal{B}^{\vec{p}_{1},\lambda_{1}}(\mathbb{R}^n)}\|\chi_{2^{k+1}B}\|_{L^{\vec{s}}(\mathbb{R}^n)}
\|\chi_{2^{k+1}B}\|_{L^{\vec{p_{1}}}(\mathbb{R}^n)}\\
&\leq|B|^{\lambda}\sum\limits_{k=1}^{\infty}k2^{kn\lambda}\|b\|_{{CBMO}^{\vec{p}_{2},\lambda_{2}}(\mathbb{R}^n)}
\|f\|_{\mathcal{B}^{\vec{p}_{1},\lambda_{1}}(\mathbb{R}^n)}\\
&\lesssim|B|^{\lambda}\|b\|_{{CBMO}^{\vec{p}_{2},\lambda_{2}}(\mathbb{R}^n)}\|f\|_{\mathcal{B}^{\vec{p}_{1},\lambda_{1}}(\mathbb{R}^n)},
\end{align*}
where
\begin{align*}
&|b_{2^{k+1}B}-b_{B}|\leq\sum\limits_{j=0}^{k}|b_{2^{j+1}B}-b_{2^{j}B}|\\
&\leq\sum\limits_{j=0}^{k}\frac{1}{|2^{j}B|}\int_{2^{j}B}|b(y)-b_{2^{j+1}B}|dy\\
&\lesssim\sum\limits_{j=0}^{k}\frac{1}{|2^{j}B|}\|(b-b_{2^{j+1}B})\chi_{2^{j+1}B}\|_{L^{\vec{p}_{2}}(\mathbb{R}^n)}
\|\chi_{2^{j+1}B}\|_{L^{\vec{p_{2}}'}(\mathbb{R}^n)}\\
&\lesssim\|b\|_{{CBMO}^{\vec{p}_{2},\lambda_{2}}(\mathbb{R}^n)}\sum\limits_{j=0}^{k}\frac{1}{|2^{j}B|}|2^{j+1}B|^{\lambda_{2}}
\|\chi_{2^{j+1}B}\|_{L^{\vec{p}_{2}}(\mathbb{R}^n)}\|\chi_{2^{j+1}B}\|_{L^{\vec{p_{2}}'}(\mathbb{R}^n)}\\
&\lesssim\|b\|_{{CBMO}^{\vec{p}_{2},\lambda_{2}}(\mathbb{R}^n)}\sum\limits_{j=0}^{k}|2^{j+1}B|^{\lambda_{2}}\\
&\lesssim\|b\|_{{CBMO}^{\vec{p}_{2},\lambda_{2}}(\mathbb{R}^n)}(k+1)|2^{k+1}B|^{\lambda_{2}}.
\end{align*}
Thus,
\begin{align*}
J_4&=\|T_{\alpha}((b-b_B)f_2)\chi_{B}\|_{L^{\vec{q}}(\mathbb{R}^n)}=\|T_{\alpha}((b-b_B)f\chi_{(2B)^{c}})\chi_{B}\|_{L^{\vec{q}}(\mathbb{R}^n)}\\
&\lesssim|B|^{\lambda}\|\chi_{B}\|_{L^{\vec{q}}(\mathbb{R}^n)}\|b\|_{{CBMO}^{\vec{p}_{2},\lambda_{2}}(\mathbb{R}^n)}
\|f\|_{\mathcal{B}^{\vec{p}_{1},\lambda_{1}}(\mathbb{R}^n)}.
\end{align*}
Combining all of the above estimates, we complete the proof of Theorem \ref{th3.2}.
\end{proof}
\section{Generalized mixed central Morrey spaces}

In fact, we can further extend the boundedness of  the operators considered by in Section 3 to the generalized mixed central Morrey spaces.  For  this to happen, we first recall the definition of the generalized mixed Morrey spaces $M^{\varphi}_{\vec{q}}(\mathbb{R}^n)$.

\quad\hspace{-22pt}{\bf Definition 4.1}(\cite{WM2}) {For $1\leq\vec{q}<\infty$ and $\varphi(x,r):\mathbb{R}^{n}\times(0.\infty)\rightarrow(0.\infty)$ is a Lebesgue measurable function, the  generalized mixed Morrey spaces $M^{\varphi}_{\vec{q}}(\mathbb{R}^n)$ are defined by
$$M^{\varphi}_{\vec{q}}(\mathbb{R}^n):=\{f\in\mathscr{M}(\mathbb{R}^{n}):\|f\|_{M^{\varphi}_{\vec{q}}(\mathbb{R}^n)}<\infty\},$$
where
$$\|f\|_{M^{\varphi}_{\vec{q}}(\mathbb{R}^n)}:=\sup\limits_{x\in\mathbb{R}^n, r>0}\varphi(x,r)^{-1}\|\chi_{B(x,r)}\|^{-1}_{L^{\vec{q}}(\mathbb{R}^{n})}
\|f\chi_{B(x,r)}\|_{L^{\vec{q}}(\mathbb{R}^{n})}.$$

Similarly, we also can give the definition of the generalized mixed central Morrey spaces as follows:

\quad\hspace{-22pt}{\bf Definition 4.2} {Let $\varphi(r)$ be a positive measurable function on $\mathbb{R}^+$, $\lambda\in\mathbb{R}$ and $1<\vec{q}<\infty$. The generalized mixed central Morrey spaces $\mathcal{B}^{\vec{q},\varphi}(\mathbb{R}^n)$ is defined by
$$\mathcal{B}^{\vec{q},\varphi}(\mathbb{R}^n)=\big\{f\in L_{loc}^q(\mathbb{R}^n): \|f\|_{\mathcal{B}^{\vec{q},\varphi}(\mathbb{R}^n)}<\infty\big\},$$
where
$$\|f\|_{\mathcal{B}^{\vec{q},\varphi}(\mathbb{R}^n)}:=\sup\limits_{r>0}\frac{1}{\varphi(r)}
\frac{\|f\chi_{B(0,r)}\|_{L^{\vec{q}}(\mathbb{R}^n)}}{\|\chi_{B(0,r)}\|_{L^{\vec{q}}(\mathbb{R}^n)}}.$$
}
\quad\hspace{-22pt}{\bf Remark 4.1} {\it Note that if we take $\varphi(r)=r^{n\lambda}$, then $\mathcal{B}^{\vec{q},\varphi}(\mathbb{R}^n)=\mathcal{B}^{\vec{q},\lambda}(\mathbb{R}^n)$.}

Now we extend Theorems \ref{th3.1}-\ref{th3.2} to the generalized mixed central Morrey space $\mathcal{B}^{\vec{q},\varphi}(\mathbb{R}^n)$.

\begin{theorem} \label{th4.1}
Let $0\leq\alpha<n,1<\vec{q}<\infty$ satisfy conditions $\vec{p}<\frac{n}{\alpha}$ and $\alpha=\sum^{n}_{i=1}\frac{1}{p_{i}}-\sum^{n}_{i=1}\frac{1}{q_{i}}$. If the pair $(\varphi_1,\varphi_2)$ satisfy the condition
$$\int_{r}^{\infty}\frac{ess\inf_{t<\tau<\infty}\varphi_1(\tau){\tau}^{\sum^{n}_{i=1}\frac{1}{p_{i}}}}{{t^{\sum^{n}_{i=1}\frac{1}{p_{i}}+1}}}dt\lesssim\varphi_2(r).$$
then the fractional integral operator $T_{\alpha}$ is bounded from $\mathcal{B}^{\vec{p},\varphi_{1}}(\mathbb{R}^n)$ to $\mathcal{B}^{\vec{q},\varphi_{2}}(\mathbb{R}^n)$. Moreover,
$$\|T_{\alpha}f\|_{\mathcal{B}^{\vec{q},\varphi_{2}}(\mathbb{R}^n)}\lesssim\|f\|_{\mathcal{B}^{\vec{p},\varphi_{1}}(\mathbb{R}^n)}.$$
\end{theorem}

\begin{theorem} \label{th4.2}
Let $0\leq\alpha<n,1<\vec{p}_{1},\vec{p_{2}},\vec{q}<\infty$ satisfy conditions $\vec{p_{1}}<\frac{n}{\alpha}, {\vec{p}_{1}}'<\vec{p}_{2}$ and $\alpha=\sum^{n}_{i=1}\frac{1}{p_{1i}}+\sum^{n}_{i=1}\frac{1}{p_{2i}}-\sum^{n}_{i=1}\frac{1}{q_{i}}$. Let $0<\lambda<\frac{1}{n}$ and
$b\in CBMO^{\vec{p}_{2},\lambda}(\mathbb{R}^{n})$. If the pair $(\varphi_1,\varphi_2)$ satisfy the condition
$$\int_{r}^{\infty}t^{n\lambda}(1+\ln\frac{t}{r})
\frac{ess\inf_{t<\tau<\infty}\varphi_1(\tau){\tau}^{\sum^{n}_{i=1}\frac{1}{p_{1i}}+\alpha}}{{t^{\sum^{n}_{i=1}\frac{1}{p_{1i}}+1}}}dt\lesssim\varphi_2(r).$$
then the commutator $[b,T_{\alpha}]$ is bounded from $\mathcal{B}^{\vec{p}_{1},\varphi_{1}}(\mathbb{R}^n)$ to $\mathcal{B}^{\vec{q},\varphi_{2}}(\mathbb{R}^n)$. Moreover,
$$\|[b,T_{\alpha}]f\|_{\mathcal{B}^{\vec{q},\varphi_{2}}(\mathbb{R}^n)}\lesssim\|b\|_{CBMO^{\vec{p}_{2},\lambda}}\|f\|_{\mathcal{B}^{\vec{p}_{1},\varphi_{1}}(\mathbb{R}^n)}.$$
\end{theorem}
The proofs of Theorems \ref{th4.1} and \ref{th4.2} are based on the following two key lemmas. For this purpose, we provide the details as follows.
\begin{lemma}\label{le4.1}
Let $0\leq\alpha<n,1<\vec{q}<\infty$ satisfy conditions $\vec{p}<\frac{n}{\alpha}$ and $\alpha=\sum^{n}_{i=1}\frac{1}{p_{i}}-\sum^{n}_{i=1}\frac{1}{q_{i}}$, then the inequality
$$\|T_{\alpha}(f)\|_{L^{\vec{q}}(B(0,r))}\lesssim r^{\sum^{n}_{i=1}\frac{1}{q_{i}}}\int_{2r}^{\infty}t^{-\sum^{n}_{i=1}\frac{1}{q_{i}}-1}\|f\|_{L^{\vec{p}}(B(0,t))}dt$$
holds for any ball $B(0,r)$ and for all $f\in L^{\vec{p}}_{loc}(\mathbb{R}^n)$.
\end{lemma}

\begin{proof}[Proof]
For any $r>0$, set $B=B(0,r)$ and $2B=B(0,2r)$, we write
$$f(x)=f(x){\chi}_{2B}(x)+f(x){\chi}_{(2B)^{c}}(x)=:f_1(x)+f_2(x)$$
and have
\begin{align*}
\|T_{\alpha}(f)\|_{L^{\vec{q}}(B)}&\leq\|T_{\alpha}(f_1)\|_{L^{\vec{q}}(B)}+\|T_{\alpha}(f_2)\|_{L^{\vec{q}}(B)}\\
&=:I_1+I_2.
\end{align*}
For $I_1$, since $f_1\in L^{\vec{p}}(\mathbb{R}^n)$ and $T_{\alpha}$ is bounded from $L^{\vec{p}}(\mathbb{R}^n)$ to $L^{\vec{q}}(\mathbb{R}^n)$, we get
\begin{align*}
I_1=\|T_{\alpha}(f_1)\|_{L^{\vec{q}}(B)}&\leq\|T_{\alpha}(f_1)\|_{L^{\vec{q}}(\mathbb{R}^n)}\lesssim\|f_1\|_{L^{\vec{p}}(\mathbb{R}^n)}
=\|f\chi_{2B}\|_{L^{\vec{p}}(\mathbb{R}^n)}\\
&\lesssim\|f\|_{L^{\vec{p}}(B(0,2r))}r^{\sum^{n}_{i=1}\frac{1}{q_{i}}}\int_{2r}^{\infty}t^{-\sum^{n}_{i=1}\frac{1}{q_{i}}-1}dt\\
&\leq r^{\sum^{n}_{i=1}\frac{1}{q_{i}}}\int_{2r}^{\infty}t^{-\sum^{n}_{i=1}\frac{1}{q_{i}}-1}\|f\|_{L^{\vec{p}}(B(0,t))}dt.
\end{align*}

We now estimate $I_2$. It is clear that $x\in B$, and $y\in{(2B)^{c}}$, then by the Fubini theorem and the fact that $|y-x|\geq|y|-|x|\geq\frac{1}{2}|y|$ lead to the following
\begin{align*}
|T_{\alpha}(f_2)(x)|&\leq\int_{(2B)^{c}}\frac{|f(y)|}{|x-y|^{n-\alpha}}dy\\
&\lesssim\int_{(2B)^{c}}\frac{|f(y)|}{|y|^{n-\alpha}}dy\\
&\lesssim\int_{(2B)^{c}}|f(y)|\int^{\infty}_{|y|}\frac{1}{t^{n-\alpha+1}}dtdy\\
&\leq\int^{\infty}_{2r}\int_{2r<|y|<t}|f(y)|\frac{1}{t^{n-\alpha+1}}dydt\\
&\lesssim\int^{\infty}_{2r}\|f\|_{L^{\vec{p}}(B(0,t))}t^{n-\sum^{n}_{i=1}\frac{1}{p_{i}}}\frac{1}{t^{n-\alpha+1}}dt.
\end{align*}
Thus, we have
$$I_2=\|T_{\alpha}(f_2)\|_{L^{\vec{q}}(B)}\lesssim r^{\sum^{n}_{i=1}\frac{1}{q_{i}}}\int_{2r}^{\infty}t^{-\sum^{n}_{i=1}\frac{1}{q_{i}}-1}\|f\|_{L^{\vec{p}}(B(0,t))}dt.$$
Combining the estimates $I_1$ and $I_2$, the proof of Lemma \ref{le4.1} is completed.
\end{proof}

\begin{lemma}\label{le4.2}
Let $0\leq\alpha<n,1<\vec{p}_{1},\vec{p_{2}},\vec{q}<\infty$ satisfy conditions $\vec{p_{1}}<\frac{n}{\alpha}, {\vec{p}_{1}}'<\vec{p}_{2}$ and $\alpha=\sum^{n}_{i=1}\frac{1}{p_{1i}}+\sum^{n}_{i=1}\frac{1}{p_{2i}}-\sum^{n}_{i=1}\frac{1}{q_{i}}$. Let $0<\lambda<\frac{1}{n}$ and
$b\in CBMO^{\vec{p}_{2},\lambda}(\mathbb{R}^{n})$, then the inequality
\begin{align*}
\|[b,T_{\alpha}](f)\|_{L^{\vec{q}}(B(0,r))}&\lesssim r^{\sum^{n}_{i=1}\frac{1}{q_{i}}}\|b\|_{{CBMO}^{\vec{p}_{2},\lambda}(\mathbb{R}^n)}\\
&\quad\times\int_{2r}^{\infty}t^{n\lambda}\big(1+\ln\frac{t}{r}\big)
\frac{\|f\|_{L^{\vec{p}_1}(B(0,t))}}{t^{\sum^{n}_{i=1}\frac{1}{p_{1i}}+1-\alpha}}dt
\end{align*}
holds for any ball $B(0,r)$ and for all $f\in L^{\vec{p}_1}_{loc}(\mathbb{R}^n)$.
\end{lemma}

\begin{proof}[Proof]
For $b\in{CBMO}^{p_2,\lambda}(\mathbb{R}^n)$ and any $r>0$, set $B=B(0,r)$ and $2B=B(0,2r)$, we write
$$f(x)=f(x){\chi}_{2B}(x)+f(x){\chi}_{(2B)^{c}}(x)=:f_1(x)+f_2(x),$$
and
\begin{align*}
[b,T_{\alpha}](f)(x)&=[b-b_B,T_{\alpha}](f)(x)\\
&\leq(b(x)-b_B)T_{\alpha}(f_1)(x)+T_{\alpha}((b-b_B)f_1)(x)\\
&\quad+(b(x)-b_B)T_{\alpha}(f_2)(x)+T_{\alpha}((b-b_B)f_2)(x).
\end{align*}
Hence, we have
\begin{align*}
&\|[b,T_{\alpha}](f)\|_{L^{\vec{q}}(B(0,r))}\\
&\leq\|(b(x)-b_B)T_{\alpha}(f_1)\|_{L^{\vec{q}}(B(0,r))}+\|T_{\alpha}((b-b_B)f_1)\|_{L^{\vec{q}}(B(0,r))}\\
&\quad+\|(b(x)-b_B)T_{\alpha}(f_2)\|_{L^{\vec{q}}(B(0,r))}+\|T_{\alpha}((b-b_B)f_2)\|_{L^{\vec{q}}(B(0,r))}\\
&=:J_1+J_2+J_3+J_4.
\end{align*}
For $J_1$, choose $\vec{l}$ such that $\frac{1}{\vec{q}}=\frac{1}{\vec{p_{2}}}+\frac{1}{\vec{l}}$, then $\alpha=\sum^{n}_{i=1}\frac{1}{p_{1i}}-\sum^{n}_{i=1}\frac{1}{l_{i}}$.
Using H\"{o}lder's inequality and the $(L^{\vec{p}_{1}}(\mathbb{R}^{n}),L^{\vec{l}}(\mathbb{R}^{n}))$-boundedness of $T_{\alpha}$, we get
\begin{align*}
J_1&=\|(b(x)-b_B)T_{\alpha}(f_1)\|_{L^{\vec{q}}(B(0,r))}\\
&\lesssim\|(b-b_{B})\chi_{B}\|_{L^{\vec{p}_{2}}(\mathbb{R}^{n})}\|T_{\alpha}f_1\|_{L^{\vec{l}}(\mathbb{R}^{n})}\\
&\lesssim|B|^{\lambda}\|\chi_{B}\|_{L^{\vec{p}_{2}}(\mathbb{R}^{n})}\|b\|_{CBMO^{\vec{p}_{2},\lambda}(\mathbb{R}^{n})}
\|f_1\|_{L^{\vec{p}_{1}}(\mathbb{R}^{n})}\\
&\leq r^{n\lambda+\sum^{n}_{i=1}\frac{1}{p_{2i}}}\|b\|_{CBMO^{\vec{p}_{2},\lambda}(\mathbb{R}^{n})}\|f\|_{L^{\vec{p}_1}(2B)}\\
&\lesssim r^{n\lambda+\sum^{n}_{i=1}\frac{1}{p_{2i}}+\sum^{n}_{i=1}\frac{1}{l_{i}}}\|b\|_{CBMO^{\vec{p}_{2},\lambda}(\mathbb{R}^{n})}
\int_{2r}^{\infty}\|f\|_{L^{\vec{p}_{1}}(B(0,t))}\frac{1}{t^{\sum^{n}_{i=1}\frac{1}{l_{i}}+1}}dt\\
&\leq r^{\sum^{n}_{i=1}\frac{1}{q_{i}}}\|b\|_{CBMO^{\vec{p}_{2},\lambda}(\mathbb{R}^{n})}
\int_{2r}^{\infty}t^{n\lambda}\|f\|_{L^{\vec{p}_{1}}(B(0,t))}\frac{1}{t^{\sum^{n}_{i=1}\frac{1}{p_{1i}}+1-\alpha}}dt\\
&\lesssim r^{\sum^{n}_{i=1}\frac{1}{q_{i}}}\|b\|_{CBMO^{\vec{p}_{2},\lambda}(\mathbb{R}^{n})} \int_{2r}^{\infty}t^{n\lambda}\big(1+\ln\frac{t}{r}\big)\|f\|_{L^{\vec{p}_{1}}(B(0,t))}\frac{1}{t^{\sum^{n}_{i=1}\frac{1}{p_{1i}}+1-\alpha}}dt.
\end{align*}
Similarly, for $J_2$, take $\vec{s}$ such that $\frac{1}{\vec{s}}=\frac{1}{\vec{p_{1}}}+\frac{1}{\vec{p}_{2}}$, then $\alpha=\sum^{n}_{i=1}\frac{1}{s_{i}}-\sum^{n}_{i=1}\frac{1}{q_{i}}$. Nothing $\vec{p_{1}}<\frac{n}{\alpha}$, then using H\"{o}lder's inequality and the $(L^{\vec{s}}(\mathbb{R}^{n}),L^{\vec{q}}(\mathbb{R}^{n}))$-boundedness of $T_{\alpha}$, we also have

\begin{align*}
J_2&=\|T_{\alpha}((b-b_B)f_1)\|_{L^{\vec{q}}(B(0,r))}\\
&\lesssim\|(b-b_B)f\|_{L^{\vec{s}}(2B)}\\
&\lesssim\|(b-b_B)\chi_{2B}\|_{L^{\vec{p}_{2}}(\mathbb{R}^{n})}\|f\|_{L^{\vec{p}_{1}}(2B)}\\
&\lesssim|2B|^{\lambda}\|\chi_{2B}\|_{L^{\vec{p}_{2}}(\mathbb{R}^n)}\|b\|_{{CBMO}^{\vec{p}_{2},\lambda}(\mathbb{R}^n)}\|f\|_{L^{\vec{p}_{1}}(2B)}
\end{align*}
\begin{align*}
&\lesssim r^{n\lambda+\sum^{n}_{i=1}\frac{1}{p_{2i}}+\sum^{n}_{i=1}\frac{1}{p_{1i}}-\alpha}\|b\|_{CBMO^{\vec{p}_{2},\lambda}(\mathbb{R}^{n})}
\int_{2r}^{\infty}\|f\|_{L^{\vec{p}_{1}}(B(0,t))}\frac{1}{t^{\sum^{n}_{i=1}\frac{1}{p_{1i}}+1-\alpha}}dt\\
&\leq r^{\sum^{n}_{i=1}\frac{1}{q_{i}}}\|b\|_{CBMO^{\vec{p}_{2},\lambda}(\mathbb{R}^{n})}
\int_{2r}^{\infty}t^{n\lambda}\|f\|_{L^{\vec{p}_{1}}(B(0,t))}\frac{1}{t^{\sum^{n}_{i=1}\frac{1}{p_{1i}}+1-\alpha}}dt\\
&\lesssim r^{\sum^{n}_{i=1}\frac{1}{q_{i}}}\|b\|_{CBMO^{\vec{p}_{2},\lambda}(\mathbb{R}^{n})} \int_{2r}^{\infty}t^{n\lambda}\big(1+\ln\frac{t}{r}\big)\|f\|_{L^{\vec{p}_{1}}(B(0,t))}\frac{1}{t^{\sum^{n}_{i=1}\frac{1}{p_{1i}}+1-\alpha}}dt.
\end{align*}
For $J_3$, by Lemma \ref{le4.1}, we know that
$$|T_{\alpha}(f_2)(x)|\lesssim\int^{\infty}_{2r}\|f\|_{L^{\vec{p}_{1}}(B(0,t))}t^{n-\sum^{n}_{i=1}\frac{1}{p_{1i}}}\frac{1}{t^{n-\alpha+1}}dt,$$
which, for $\frac{1}{\vec{q}}=\frac{1}{\vec{p_{2}}}+\frac{1}{\vec{l}}$, together with H\"{o}lder's inequality and $J_1$, implies that
\begin{align*}
J_3&\lesssim\|b-b_{B}\|_{L^{\vec{q}}(B)}\int^{\infty}_{2r}\|f\|_{L^{\vec{p}_{1}}(B(0,t))}
t^{n-\sum^{n}_{i=1}\frac{1}{p_{1i}}}\frac{1}{t^{n-\alpha+1}}dt\\
&\lesssim\|b-b_{B}\chi_{B}\|_{L^{\vec{p}_{2}}(\mathbb{R}^{n})}\|\chi_{B}\|_{L^{\vec{l}}(B)}
\int^{\infty}_{2r}\|f\|_{L^{\vec{p}_{1}}(B(0,t))}\frac{1}{t^{\sum^{n}_{i=1}\frac{1}{p_{1i}}+1-\alpha}}dt\\
&\lesssim r^{\sum^{n}_{i=1}\frac{1}{q_{i}}}\|b\|_{CBMO^{\vec{p}_{2},\lambda}(\mathbb{R}^{n})} \int_{2r}^{\infty}t^{n\lambda}\big(1+\ln\frac{t}{r}\big)\|f\|_{L^{\vec{p}_{1}}(B(0,t))}\frac{1}{t^{\sum^{n}_{i=1}\frac{1}{p_{1i}}+1-\alpha}}dt.
\end{align*}

We now estimate $J_4$. It is clear that $x\in B$, and $y\in{(2B)^{c}}$, then by the fact that $|y-x|\geq|y|-|x|\geq\frac{1}{2}|y|$ leads to the following result
\begin{align*}
|T_{\alpha}((b-b_B)f\chi_{(2B)^{c}})(x)|&\leq\int_{(2B)^{c}}\frac{|b(y)-b_B||f(y)|}{|x-y|^{n-\alpha}}dy\\
&\lesssim\int_{(2B)^{c}}\frac{|b(y)-b_B||f(y)|}{|y|^{n-\alpha}}dy,
\end{align*}
by taking $\|\cdot\|_{L^{\vec{q}}(B(0,r))}$ and the Fubini theory, we get
\begin{align*}
J_4&\lesssim\int_{(2B)^{c}}\frac{|b(y)-b_B||f(y)|}{|y|^{n-\alpha}}dy\|\chi_{B}\|_{L^{\vec{q}}(\mathbb{R}^{n})}\\
&\leq r^{\sum^{n}_{i=1}\frac{1}{q_{i}}}\int_{(2B)^{c}}\frac{|b(y)-b_B||f(y)|}{|y|^{n-\alpha}}dy\\
&\lesssim r^{\sum^{n}_{i=1}\frac{1}{q_{i}}}\int_{(2B)^{c}}|b(y)-b_B||f(y)|\int^{\infty}_{|y|}\frac{1}{t^{n-\alpha+1}}dtdy\\
&\leq r^{\sum^{n}_{i=1}\frac{1}{q_{i}}}\int^{\infty}_{2r}\int_{B(0,t)}|b(y)-b_B||f(y)|dy\frac{1}{t^{n-\alpha+1}}dt.
\end{align*}
The inner integral can be estimated as follows,

\begin{align*}
\int_{B(0,t)}|b(y)-b_B||f(y)|dy&\leq\int_{B(0,t)}|b(y)-b_B(0,t)||f(y)|dy\\
&\quad+\int_{B(0,t)}|b_B(0,t)-b_B(0,r)||f(y)|dy\\
&=:J_{41}+J_{42}.
\end{align*}

For $J_{41}$, take $\vec{m}$ such that $\frac{1}{\vec{p_{1}}'}=\frac{1}{\vec{p_{2}}}+\frac{1}{\vec{m}}$, by the H\"{o}lder inequality, we have
\begin{align*}
&J_{41}\lesssim\|(b-b_{B(0,t)})\chi_{B(0,t)}\|_{L^{\vec{p}_{1}'}(\mathbb{R}^{n})}\|f\|_{L^{\vec{p}_{1}}(B(0,t))}\\
&\lesssim\|(b-b_{B(0,t)})\chi_{B(0,t)}\|_{L^{\vec{p}_{2}}(\mathbb{R}^{n})}\|\chi_{B(0,t)}\|_{L^{\vec{m}}(\mathbb{R}^{n})}
\|f\|_{L^{\vec{p}_{1}}(B(0,t))}\\
&\lesssim t^{n\lambda+\sum^{n}_{i=1}\frac{1}{p'_{1i}}}
\|b\|_{{CBMO}^{\vec{p}_{2},\lambda}(\mathbb{R}^n)}\|f\|_{L^{\vec{p}_{1}}(B(0,t))}.
\end{align*}

For $J_{42}$, by the H\"{o}lder inequality and Proposition \ref{pro2.1}, we obtain
\begin{align*}
J_{42}&=|b_B(0,t)-b_B(0,r)|\int_{B(0,t)}|f(y)|dy\\
&\lesssim|b_B(0,t)-b_B(0,r)|\|f\|_{L^{\vec{p}_{1}}(B(0,t))}\|\chi_{B(0,t)}\|_{L^{\vec{p}_{1}'}(B(0,t))}\\
&\lesssim t^{\sum^{n}_{i=1}\frac{1}{p'_{1i}}}\|f\|_{L^{\vec{p}_{1}}(B(0,t))}\frac{1}{|B(0,t)|}\|(b-b_{B(0,r)}\chi_{B(0,t)}\|_{L^{1}(\mathbb{R}^{n})}\\
&\lesssim t^{\sum^{n}_{i=1}\frac{1}{p'_{1i}}}\|f\|_{L^{\vec{p}_{1}}(B(0,t))}
\frac{1}{|B(0,t)|}\|(b-b_{B(0,r)}\chi_{B(0,t)}\|_{L^{\vec{p}_{2}}(\mathbb{R}^{n})}\|\chi_{B(0,t)}\|_{L^{\vec{p}_{2}'}(\mathbb{R}^{n})}\\
&\lesssim t^{n\lambda}(1+\ln\frac{t}{r})t^{\sum^{n}_{i=1}\frac{1}{p'_{1i}}}\|b\|_{{CBMO}^{\vec{p}_{2},\lambda}(\mathbb{R}^n)}\|f\|_{L^{\vec{p}_{1}}(B(0,t))}.
\end{align*}
Combining $J_{41}$ and $J_{42}$ gives that
\begin{align*}
J_4&\lesssim r^{\sum^{n}_{i=1}\frac{1}{q_{i}}}\|b\|_{{CBMO}^{\vec{p}_{2},\lambda}(\mathbb{R}^n)}\int^{\infty}_{2r}t^{n\lambda}(1+\ln\frac{t}{r})
t^{\sum^{n}_{i=1}\frac{1}{p'_{1i}}}\|f\|_{L^{\vec{p}_{1}}(B(0,t))}\frac{1}{t^{n-\alpha+1}}dt\\
&=r^{\sum^{n}_{i=1}\frac{1}{q_{i}}}\|b\|_{{CBMO}^{\vec{p}_{2},\lambda}(\mathbb{R}^n)}\int^{\infty}_{2r}t^{n\lambda}(1+\ln\frac{t}{r})
\|f\|_{L^{\vec{p}_{1}}(B(0,t))}\frac{1}{t^{\sum^{n}_{i=1}\frac{1}{p_{1i}}+1-\alpha}}dt.
\end{align*}
This completes the proof of Lemma \ref{le4.2}.
\end{proof}

Now we give the proofs of Theorems \ref{th4.1}-\ref{th4.2} in this position. In fact, the methods of the proofs are standard, so we only provide Theorem \ref{th4.2} with the following.

\begin{proof}[Proof of Theorem \ref{th4.2}]
For convenience of calculations,  put  $\sum^{n}_{i=1}\frac{1}{p_{1i}}=:N$, by Lemma \ref{le4.2} and a change of variables $t=s^{-\frac{1}{N}}$, we obtain that

\begin{align*}
&\|[b,T_{\alpha}](f)\|_{\mathcal{B}^{\vec{q},\varphi_2}(\mathbb{R}^n)}\\
&\lesssim\sup\limits_{r>0}\frac{1}{\varphi_2(r)}\frac{r^{\sum^{n}_{i=1}\frac{1}{q_{i}}}}{\|\chi_{B}\|_{L^{\vec{q}}(\mathbb{R}^{n})}}
\int_{2r}^{\infty}t^{n\lambda}(1+\ln\frac{t}{r})\|f\|_{L^{\vec{p}_{1}}(B(0,t))}\frac{1}{t^{N+1-\alpha}}dt\\
&\lesssim\sup\limits_{r>0}\frac{1}{\varphi_2(r)}\int_{0}^{r^{-N}}s^{\frac{-n\lambda-\alpha}{N}}(1+\ln\frac{s^{-\frac{1}{N}}}{r})
\|f\|_{L^{\vec{p}_{1}}(B(0,s^{-\frac{1}{N}}))}ds\\
&=\sup\limits_{r>0}\frac{1}{\varphi_2(r^{-\frac{1}{N}})}\int_{0}^{r}s^{\frac{-n\lambda-\alpha}{N}}(1+\frac{1}{N}\ln\frac{r}{s})
\|f\|_{L^{\vec{p}_{1}}(B(0,s^{-\frac{1}{N}}))}ds.
\end{align*}
If we set
$$\omega(t)={\varphi_2(t^{-\frac{1}{N}})^{-1}t}, \quad \nu(t)={\varphi_1(t^{-\frac{1}{N}})^{-1}t^{1-\frac{-n\lambda-\alpha}{N}}},$$
since the pair $(\varphi_1,\varphi_2)$ satisfy the following condition
$$\int_{r}^{\infty}t^{n\lambda}(1+\ln\frac{t}{r})
\frac{ess\inf_{t<\tau<\infty}\varphi_1(\tau){\tau}^{N+\alpha}}{{t^{N+1}}}dt\lesssim\varphi_2(r).$$
It follows that
$$\sup\limits_{t>0}\frac{\omega(t)}{t}\int_{0}^{t}\frac{dr}{ess\sup_{0<s<r}\nu(s)}<\infty.$$
This leads to the following inequality (see \cite{ML})
$$ess\sup_{t>0}\omega(t)\mathcal{H}g(t)\lesssim ess\sup_{t>0}\nu(t)g(t)$$
holds for all nonnegative and non-increasing functions $g$ on $(0,\infty)$, where $\mathcal{H}$ is the classical Hardy operator, that is,
$$\mathcal{H}g(t)=\frac{1}{t}\int_{0}^{t}g(r)dr.$$
Therefore, let $g(t)=t^{\frac{-n\lambda-\alpha}{N}}(1+\frac{1}{N}\ln\frac{r}{t})\|f\|_{L^{\vec{p}_{1}}(B(0,t^{-\frac{1}{N}}))}$, we have
\begin{align*}
\|[b,T_{\alpha}](f)\|_{\mathcal{B}^{\vec{q},\varphi_2}(\mathbb{R}^n)}
&\lesssim\sup\limits_{r>0}\frac{r}{\varphi_1(r^{-\frac{1}{N}})}\|f\|_{L^{\vec{p}_{1}}(B(0,r^{-\frac{1}{N}}))}\\
&=\sup\limits_{r>0}\frac{1}{\varphi_1(r)}r^{-\sum^{n}_{i=1}\frac{1}{p_{1i}}}\|f\|_{L^{\vec{p}_{1}}(B(0,r))}\\
&=\|f\|_{\mathcal{B}^{q,\varphi_1}(\mathbb{R}^n)}.
\end{align*}
The proof of Theorem \ref{th4.2} is completed.
\end{proof}


\textbf{Acknowledgement.}
The authors would like to express their gratitude to the referee for his/her very valuable comments.


\begin{thebibliography}{99}
\bibitem{AI}{N. Antonic  and I. Ivec},
On the H\"{o}rmander-Mihlin theorem for mixed-norm Lebesgue spaces.
Math. Anal. Appl. {\bf 433} (2016), 176-199.

\bibitem{BP}{A. Benedek and R. Panzone},
The space $L^{p}$, with mixed norm.
Duke Math. {\bf 28} (1961), 301-324.

\bibitem{ML}{M. Carro, L. Pick, J. Soria and V.D. Stepanow},
On embeddings between classical Lorentz spaces. Math. Inequal. Appl. {\bf 4}(3) (2001), 397-428.

\bibitem{CGN}{G. Cleanthous, A.G. Georgiadis and M. Nielsen},
Anisotropic mixed-norm Hardy spaces.
J. Geom. Anal. {\bf 30} (2017).

\bibitem{CM}{C.B. Morrey},
On the solutions of quasi-linear elliptic partial differential equations.
Trans. Amer. Math. Soc. {\bf 43} (1938), 126-166.

\bibitem{CS}{T. Chen  and  W. Sun},
Iterated and mixed weak norms with applications to geometric inequalities.
{\bf 13} (2017), DOI: 10.1007/s12220-019-00243-x.

\bibitem{FD}{D.L. Fernandez},
Lorentz spaces, with mixed norms.
Funct. Anal. {\bf 25} (1977), 128-146.

\bibitem{GN}{A.G. Georgiadis and  M. Nielsen},
Pseudodifferential operators on mixed-norm Besov and Triebel¨CLizorkin spaces.
Math Nachr. {\bf 289} (2016), no. 16, 2019-2036.

\bibitem{JM}{J. Alvarez, M. Guzm\'{a}n-Partida and J. Lakey},
Spaces of bounded $\lambda$-central mean oscillation, Morrey spaces, and $\lambda$-central Carleson measures.
Collect. Math. {\bf 51} (2000), 1-47.

\bibitem{JS}{J. Johnsen and W. Sickel},
On the trace problem for Lizorkin-Triebel spaces with mixed norm.
Math Nachr. {\bf 281} (2008), no. 5, 669-696.

\bibitem{KC}{C.E. Kenig},
On the local and global well-posedness theory for the KP-I equation.
Ann. Inst. H. Poincar Anal. Non Linaire {\bf 21} (2004), 827-838.

\bibitem{KD}{D. Kim},
Elliptic and parabolic equations with measurable coefficients in $L^{p}$-spaces with mixed norms.
Methods Appl. Anal. {\bf 15} (2008), 437-468.

\bibitem{KN}{N.V. Krylov},
Parabolic equations with VMO coefficients in Sobolev spaces with mixed norms.
Funct. Anal. {\bf 250} (2007), 521-558.

\bibitem{LY2}{S. Lu and D. Yang},
The central BMO spaces and Littlewood-Paley operators.
Approx. Theory Appl. {\bf 11} (1995), 72-94.

\bibitem{MM1}{M. Milman},
Embeddings of Lorentz-Marcinkiewicz spaces with mixed norms.
Anal. Math. {\bf 4} (1978), 215-223.

\bibitem{MM2}{M. Milman},
A note on $L^{(p,q)}$ spaces and Orlicz spaces with mixed norms.
Proc. Amer. Math. {\bf 83} (1981), 743-746.

\bibitem{SE}{E.M. Stein},
Harmonic Analysis: real variable methods, orthogonality, and oscillatory integrals.
Princeton University Press. (1993).

\bibitem{TM}{T. Mizuhara},
Boundedness of some classical operators on generalized Morrey spaces.
Conf. Proc. {\bf 90} (1990), 183-189.

\bibitem{TN1}{T. Nogayama},
Mixed Morrey spaces.
Positivity. {\bf 23} (2019), no. 4, 961-1000.

\bibitem{TN2}{T. Nogayama},
Boundedness of commutators of fractional integral operators on mixed Morrey spaces.
Integral Transform. Spec. Funct. {\bf 30} (2019), no. 10, 790-816.

\bibitem{WM1}{M. Wei},
Estimates for weighted Hardy-Littlewood averages and their commutators on mixed central Morrey spaces.
J. Math. Inequal. {\bf 16} (2022), 659-669.

\bibitem{WM2}{M. Wei},
Boundedness criterion for sublinear operators and commutators on generalized mixed Morrey spaces.
(2021) DOI:10.48550/arXiv.2106.12872, Preprint.

\bibitem{YT}{X. Yu and X. Tao},
Boundedness for a class of generalized commutators on $\lambda$-central Morrey spaces.
Acta Math. Sin. {\bf 29}(10) (2013), 1917-1926.
\end{thebibliography}
\end{document}